\documentclass[letterpaper]{amsart}
\pdfoutput=1
\usepackage{amsmath}
\usepackage{amsfonts}
\usepackage{amssymb}
\usepackage{amsthm}
\usepackage{graphicx}
\usepackage{bbm}
\usepackage{hyperref}
\usepackage{cite}

\def\thetitle{Hopf fibrations are characterized by being fiberwise homogeneous}
\def\theauthor{Haggai Nuchi}

\hypersetup{pdftitle=\thetitle, pdfauthor=\theauthor, colorlinks=true, linkcolor=black, citecolor=black, filecolor=black, urlcolor=black}

\newcommand{\R}{\mathbb{R}}
\newcommand{\Z}{\mathbb{Z}}
\newcommand{\C}{\mathbb{C}}
\newcommand{\HH}{\mathbb{H}}
\newcommand{\F}{\mathcal{F}}
\newcommand{\mychi}{\ensuremath \raisebox{2pt}{$\chi$}}
\newcommand{\Tr}{\operatorname{Tr}}
\newcommand{\Isom}{\operatorname{Isom}}
\newcommand{\Gr}{G_2\R^4}
\newcommand{\n}{\nabla}

\newtheorem{thm}{Theorem}[section]
\newtheorem*{thm:local}{Theorem~\ref{Thm:LocallyFWH}}
\newtheorem{prop}[thm]{Proposition}
\newtheorem{lem}[thm]{Lemma}
\newtheorem{cor}[thm]{Corollary}
\newtheorem*{mainthm}{Main Theorem}
\newtheorem*{keyLem}{Key Lemma}
\newcommand{\keylemma}{%
  \begin{keyLem}
      Let $G$ be a compact connected Lie group acting irreducibly on $\R^d$, and acting transitively on a fiberwise homogeneous $C^1$ fibration $\F$ of $S^{d-1}$ by spheres. Let $H\vartriangleleft G$ be the normal subgroup (possibly discrete or disconnected) of $G$ which takes each fiber of $\F$ to itself.
      \begin{enumerate}
        \item Suppose $G$ acts on $\R^{2n+2}$, and $G/H\cong SU(n+1)/\Z_{n+1}$. Then $G$ contains $SU(n+1)$ as a subgroup, acting in the standard way on $\R^{2n+2}$.
        \item Suppose $G$ acts on $\R^{4n+4}$, and $G/H\cong Sp(n+1)/\Z_2$. Then $G$ contains $Sp(n+1)$ as a subgroup, acting in the standard way on $\R^{4n+4}$.
        \item Suppose $G$ acts on $\R^{16}$ and $G/H\cong SO(9)$. Then $G$ contains $Spin(9)$ as a subgroup, acting as the spin representation on $\R^{16}$.
      \end{enumerate}
    \end{keyLem}}

\theoremstyle{definition}
\newtheorem{dfn}[thm]{Definition}

\theoremstyle{remark}

\title[\small \thetitle]{\LARGE \thetitle}
\author{\theauthor}
\date{}

\begin{document}
  \begin{abstract}
    The Hopf fibrations of spheres by great spheres have a number of interesting properties. In particular, each one is fiberwise homogeneous: for any two great $k$-sphere fibers in the Hopf fibration of the $n$-sphere, there is a fiber-preserving isometry of the $n$-sphere taking the first given fiber to the second. In this paper, we prove that the Hopf fibrations are characterized by this property, among all fibrations of round spheres by smooth subspheres.
    
    In the special case of the 3-sphere fibered by great circles, we prove something stronger. We prove that a fibration of a connected open set by great circles which is locally fiberwise homogeneous is part of a Hopf fibration.
  \end{abstract}
  \maketitle
  
  \section{Introduction}
    \subsection{Background}
      Heinz Hopf's famous fibrations \cite{hopf1931abbildungen, hopf1935abbildungen} of $S^{2n+1}$ by great circles, $S^{4n+3}$ by great $3$-spheres, and $S^{15}$ by great $7$-spheres have a number of interesting properties. Besides providing the first examples of homotopically nontrivial maps from one sphere to another sphere of lower dimension, they all share two striking features:
	  \begin{enumerate}
	    \item Their fibers are parallel, in the sense that any two fibers are a constant distance apart, and
	    \item The fibrations are highly symmetric. For example, there is a fiber-preserving isometry of each total space which takes any given fiber to any other one.
      \end{enumerate}

      Hopf fibrations have been characterized up to isometry by the first property above, initially among all fibrations of spheres by great subspheres \cite{wong1961isoclinic, ranjan1985riemannian, escobales1975riemannian, wolf1963elliptic, wolf1963geodesic}, and later in the stronger sense among all fibrations of spheres by smooth subspheres \cite{gromoll1988low, wilking2001index}.

      In this paper, we show that the Hopf fibrations are also characterized by their ``fiberwise homogeneity'' expressed above in (2), and in the strong sense among all fibrations of spheres by smooth subspheres.

	  The proof uses the representation theory of Lie groups, and relies on the work of Montgomery-Samelson \cite{montgomery1943transformation} and Borel \cite{borel1949some}, and its generalization by Oniscik \cite{onishchik1963transitive}, in which they find all the compact Lie groups which act transitively and effectively on spheres and projective spaces.
    
      \begin{dfn}
        Let $\F$ be a fibration of a riemannian manifold $(M,g)$. We say that $\F$ is {\em fiberwise homogeneous} if for any two fibers there is an isometry of $(M,g)$ taking fibers to fibers and taking the first given fiber to the second given fiber.
      \end{dfn}
    
      In Section~\ref{Sec:HopfMainThm}, we prove the main theorem of this paper: that the Hopf fibrations are characterized by being fiberwise homogeneous. We leave the representation theory to Section~\ref{Sec:RepnTheory}.
    
      In the special case of the 3-sphere fibered by great circles, more is true. If even a connected open set is fibered by great circles in a locally fiberwise homogeneous way, then that fibration is part of a Hopf fibration. We prove this (and give a complete definition of ``locally fiberwise homogeneous'') in Section~\ref{Sec:Local}.
      
      In a companion paper \cite{nuchi2014space}, we give complete descriptions of all fiberwise homogeneous fibrations of Euclidean and Hyperbolic 3-space by geodesics, see Figure~\ref{Fig:ExampleFWH}. In another companion paper \cite{nuchi2014surprising}, we describe a surprising example of a fiberwise homogeneous fibration of the Clifford torus $S^3\times S^3$ in the 7-sphere by great 3-spheres, which is not part of a Hopf fibration.
      \begin{figure}[h]
        \centering
        \includegraphics[width=0.3\textwidth]{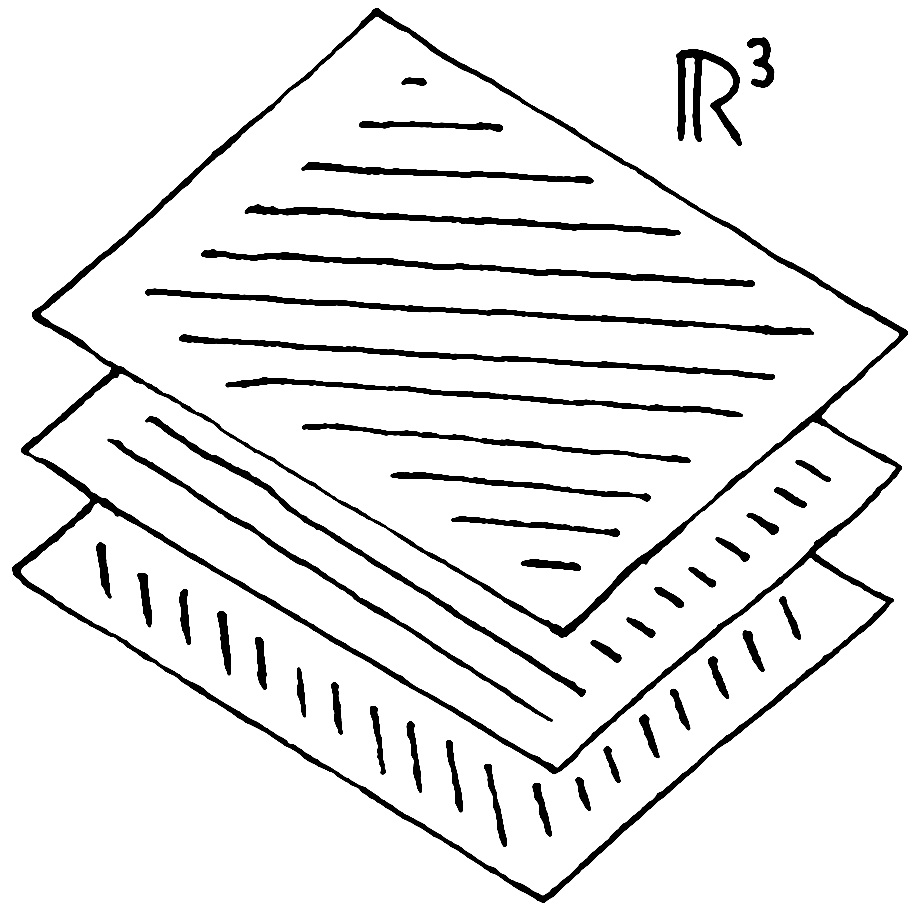}
        \caption{An example of a fiberwise homogeneous fibration of Euclidean 3-space by straight lines, not all parallel to one another. Layer 3-space by horizontal planes, and fill each plane by parallel lines, with the angle changing at a constant rate with respect to the height of the plane.}
        \label{Fig:ExampleFWH}
      \end{figure}
      
    \subsection{Acknowledgments}
      This article is an extended version of a portion of my doctoral dissertation. I am very grateful to my advisor Herman Gluck, without whose encouragement, suggestions, and just overall support and cameraderie, this would never have been written.
      
      Thanks as well to the Math Department at the University of Pennsylvania for their support during my time there as a graduate student.
    
  \section{Main Theorem}\label{Sec:HopfMainThm}
    \begin{mainthm}\label{Thm:HopfFWH}
      Let $\F$ be a fiberwise homogeneous $C^1$ fibration of $S^n$ with its standard metric by subspheres of dimension $k$, i.e.
      \begin{itemize}
        \item $S^3,S^5,S^7,\ldots$ by circles
        \item $S^7,S^{11},S^{15},\ldots$ by 3-spheres
        \item $S^{15}$ by 7-spheres.
      \end{itemize}
      Then $\F$ is a Hopf fibration.
    \end{mainthm}
        
    Here is a summary of the proof. First we use the classification of homogeneous spaces to show that the base space of a fiberwise homogeneous fibration $\F$ is diffeomorphic to the base of a Hopf fibration with the same total space and the same fiber dimension. We call this the \textbf{\textit{Hopf model}} for $\F$. Then we use a theorem of Montgomery-Samelson and Borel, and generalized by Oniscik, to find a list of the compact Lie groups which can act transitively and effectively on the base space of $\F$. We find that these groups (or covers of them) are isomorphic to subgroups of the symmetry group of the Hopf model for $\F$. We then prove a lemma that the symmetry group of $\F$ acts irreducibly on the Euclidean space in which the ambient round sphere is embedded. Following that, we use the representation theory of compact Lie groups to show that the action of the symmetry group is standard; i.e., is the same as the action of the symmetry group of the Hopf model. Finally we show that $\F$ must actually have a fiber in common with its Hopf model, and hence must be identical to it.
    
    \begin{prop}[Follows from Theorem 7.50 in Besse \cite{besse1978manifolds}, citing work of Wang, Borel, and Singh Varma]\label{Prop:BaseDiff}
      Let $\F$ be a fiberwise homogeneous $C^1$ fibration of $S^n$ by $k$-spheres. Then the base of $\F$ is diffeomorphic to the base of the Hopf fibration of the same dimension.
    \end{prop}
    \begin{proof}
      The cohomology ring of the base of $\F$ is identical to that of its Hopf model. This follows from the Serre spectral sequence; see \cite[Example 1.15]{hatcher2004spectral} for a sample computation. The base of $\F$ also has the structure of a homogeneous space. By Theorem 7.50 in Besse \cite{besse1978manifolds}, the base of $\F$ is diffeomorphic to the base of its Hopf model.
    \end{proof}
    \begin{prop}[\cite{onishchik1963transitive}, Theorem 6 parts (a), (b), and (f), and Table 2]\label{Prop:EffAct}
      Let $G$ be a compact connected Lie group acting transitively and effectively on $M$.
      \begin{itemize}
        \item Let $M=\C P^n$. Then $G=SU(n+1)/\Z_{n+1}$ or possibly $G=Sp(k+1)/\Z_2$ if $n=2k+1$.
        \item Let $M=\HH P^n$. Then $G=Sp(n+1)/\Z_2$.
        \item Let $M=S^8$. Then $G=SO(9)$.
      \end{itemize}
      Notice that, if $M=S^2$ ($=\C P^1$) then $G=SO(3)$ ($=SU(2)/\Z_2 = Sp(1)/\Z_2$), and if $M=S^4$ ($=\HH P^1$) then $G=SO(5)$ ($=Sp(2)/\Z_2$).
    \end{prop}
    In what follows, whenever we write that $G$ is a Lie group acting transitively on a fibration $\F$, we take $G$ to be closed and connected. We are justified in simplifying our life in this way because of the following lemma:
    \begin{lem}\label{Lem:ClosedConnected}
      Let $M$ be a connected Riemannian manifold, and let $G$ be a subgroup of $\Isom(M)$. Denote by $\overline{G}_0$ the identity component of the closure of $G$. Suppose $G$ acts transitively on a smooth fibration $\F$ of $M$. Then, $\overline{G}_0$ acts transitively on $\F$ as well.
    \end{lem}
    \begin{proof}
      Let $G$ preserve the fibers of $\F$. Let $\{g_n\}_{n=1}^{\infty} \subset G \subseteq \Isom(M)$, and let $g_n \to g\in \Isom(M)$. If each $g_n$ preserves each fiber of $\F$, then the limit $g$ clearly does as well. Thus $\overline{G}$ preserves the fibers of $\F$. Also $G\subseteq \overline{G}$, so $\overline{G}$ acts transitively on $\F$ as well.
      
      Now let $G$ be a closed disconnected Lie subgroup of $\Isom(M)$ acting transitively on $\F$. Let the manifold $B$ be the base of $\F$. Then $B$ has the structure of a connected homogeneous space, and hence is diffeomorphic to $G/H$, where $H$ is the isotropy subgroup of $G$ fixing a point. Let $G_i$ be the connected components of $G$, $G_0$ the identity component. The subgroup $H$ intersects every $G_i$, or else $B$ would be disconnected, so there are $g_i\in H\cap G_i$ for all $i$. Then $g_i H = H$, from which it follows that the image of $G_0$ intersects the image of every $G_i$ in $G/H$. But as $g_0$ ranges across all elements of $G_0$, $g_0 g_i$ ranges across all elements of $G_i$, so $g_0 g_i H = g_0 H$ and the image of $G_0$ is identical to the image of each $G_i$. Thus the image of $G_0$ is all of $G/H$. Therefore the identity component of $G$ acts transitively on $B$ and hence on $\F$.
    \end{proof}
    From now on, $G$ will always denote a closed connected Lie group.
    \begin{lem}\label{Lem:Irred}
      Let $\F$ be a fiberwise homogeneous $C^1$ fibration of $S^n$ by $k$-spheres. Let $G\subset SO(n+1)$ act transitively on $\F$. Then the action of $G$ on $\R^{n+1}$ is irreducible.
    \end{lem}
    \begin{proof}
      We count dimensions. Suppose we have a nontrivial splitting $R^{n+1}=A\oplus B$, with $A,B$ being $G$-invariant subspaces. Then WLOG we have
      \[ 1\leq \dim A \leq \frac{1}{2}(n+1).\]
      Every fiber of $\F$ must meet the unit sphere $S(A)$ in $A$. This sphere satisfies
      \[ 0 \leq \dim S(A) \leq \frac{1}{2}(n-1). \]
      But the base of $\F$ has dimension at least
      \[ \frac{1}{2}(n+1) > \dim S(A), \]
      as we see this by checking each of the possibilities for $k$ and $n$ case by case. This is a contradiction because the fibers may not intersect one another. Thus there cannot be a nontrivial splitting by $G$-invariant subspaces, and so $G$ acts irreducibly.
    \end{proof}
    Now we give a Key Lemma.
    \keylemma
    The proof only involves standard representation theory of compact Lie groups. We defer the proof of the Key Lemma to Section~\ref{Sec:RepnTheory} so that we don't get bogged down. We now split the Main Theorem into four smaller cases, and prove each separately.
    \begin{thm}\label{Thm:3Sphere}
      Let $\F$ be a fiberwise homogeneous $C^1$ fibration of the 3-sphere by circles (i.e.~1-spheres). Then $\F$ is the Hopf fibration.
    \end{thm}
    \begin{proof}
      Let $G$ be a subgroup of $SO(4)$ which acts transitively on $\F$. Let $H$ be the normal subgroup of $G$ which takes each fiber of $\F$ to itself. Then $G/H$ acts transitively and effectively on the base of $\F$. By Proposition~\ref{Prop:BaseDiff}, the base of $\F$ is diffeomorphic to $S^2$. By Proposition~\ref{Prop:EffAct}, $G/H$ is isomorphic to $SO(3)$.
      
      By Lemma~\ref{Lem:Irred}, $G$ acts irreducibly on $\R^4$. By the Key Lemma, case 1 (since $SO(3)=SU(2)/\Z_2$), $G$ contains $SU(2)=Sp(1)$ as a subgroup, acting as left multiplication by unit quaternions on $\R^4$.
      
      Let $E$ be the subbundle of $TS^3$ consisting of tangent lines to the fibers of $\F$. The group $Sp(1)$ preserves $\F$, and hence the field of tangent lines to $\F$ is left-invariant. Therefore its integral curves (i.e.~the fibers of $\F$) form a Hopf fibration. 
    \end{proof}
    \begin{thm}\label{Thm:2n+1Sphere}
      Let $\F$ be a fiberwise homogeneous $C^1$ fibration of the $2n+1$-sphere by circles, with $n\geq 2$. Then $\F$ is the Hopf fibration.
    \end{thm}
    \begin{proof}
      Let $G$ be a subgroup of $SO(2n+2)$ which acts transitively on $\F$. Let $H$ be the normal subgroup of $G$ which takes each fiber of $\F$ to itself. Then $G/H$ acts transitively and effectively on the base of $\F$. By Proposition~\ref{Prop:BaseDiff}, the base of $\F$ is diffeomorphic to $\C P^n$. By Proposition~\ref{Prop:EffAct}, $G/H$ is isomorphic to $SU(n+1)/\Z_{n+1}$ or possibly to $Sp(k+1)/\Z_2$ if $n=2k+1$. By Lemma~\ref{Lem:Irred}, $G$ acts irreducibly on $\R^{2n+2}$. Consider the two cases for $G/H$ separately.
      \begin{enumerate}
        \item Suppose first that $G/H=SU(n+1)/\Z_{n+1}$. By the Key Lemma, case 1, $G$ contains $SU(n+1)$ as a subgroup, acting on $\R^{2n+2}$ in the standard way. We use the $SU(n+1)$ action to identify $\R^{2n+2}$ with $\C^{n+1}$. Fix $x\in S^{2n+1}$, and let $F_x$ be the fiber of $\F$ through $x$. The isotropy subgroup of $SU(n+1)$ fixing $x$ is isomorphic to $SU(n)$. Since $SU(n)$ preserves $\F$ and preserves $x$, it must also preserve the tangent line to $F_x$ through $x$. But the only way $SU(n)$ preserves the tangent line is if the tangent line points in the direction of $ix$. But $x$ is arbitrary, and so for all $x\in S^{2n+1}$, the fiber through $x$ points in the direction $ix$. The trajectories of this field of tangent lines form a Hopf fibration.
        
        \item Suppose instead that $n=2k+1$ and that $G/H=Sp(k+1)/\Z_2$. By the Key Lemma, case 2, $G$ contains $Sp(k+1)$ as a subgroup, acting on $\R^{2(2k+1)+2}=\R^{4k+4}$ in the standard way. Let $x\in S^{4k+3}$ be arbitrary, and $F_x$ the fiber of $\F$ through $x$. Identify $\R^{4k+4}$ with $\HH^{k+1}$ so that $Sp(k+1)$ acts as quaternionic linear transformations. The isotropy subgroup of $Sp(k+1)$ is isomorphic to $Sp(k)$, fixes the $\R^4$ spanned by $x,xi,xj,xk$, and does not fix any vector in the orthogonal $\R^{4k}$. Let $v_x$ be a unit tangent vector to $F_x$ at $x$. The isotropy subgroup $Sp(k)$ of $x$ must fix $v_x$, so we can write $v_x = xp$ for some purely imaginary quaternion $p$. But then the group $Sp(k+1)$, preserving $\F$, takes $x$ to any other point $y$ on $S^{4k+3}$, and takes the vector $v_x$ to the vector $v_y$ tangent to $F_y$ at $y$. Thus the field $xp$ at $x$ is tangent to $\F$, and so $\F$ is in fact isometric to the Hopf fibration.
      \end{enumerate}
    \end{proof}
    \begin{thm}\label{Thm:4n+3Sphere}
      Let $\F$ be a fiberwise homogeneous $C^1$ fibration of the $4n+3$-sphere by 3-spheres. Then $\F$ is the Hopf fibration.
    \end{thm}
    \begin{proof}
      Let $G$ be a subgroup of $SO(4n+4)$ which acts transitively on $\F$. Let $H$ be the normal subgroup of $G$ which takes each fiber of $\F$ to itself. Then $G/H$ acts transitively and effectively on the base of $\F$. By Proposition~\ref{Prop:BaseDiff}, the base of $\F$ is diffeomorphic to $\HH P^n$. By Proposition~\ref{Prop:EffAct}, $G/H$ is isomorphic to $Sp(n+1)/\Z_2$. By Lemma~\ref{Lem:Irred}, $G$ acts irreducibly. By the Key Lemma, case 2, $G$ contains $Sp(n+1)$ as a subgroup, acting in the standard way on $\R^{4n+4}$.
      
      Identify $\R^{4n+4}$ with $\HH^{n+1}$ so that $Sp(n+1)$ acts by quaternionic-linear transformations. Let $x\in S^{4n+3}$ be arbitrary, and let $F_x$ be the fiber of $\F$ through $x$. Let $P_x$ be the tangent 3-plane to $F_x$ at $x$. The isotropy subgroup of $Sp(n+1)$ fixing $x$ is isomorphic to $Sp(n)$, and it also fixes the 4-dimensional subspace of $\HH^{n+1}$ spanned by $x, xi, xj, xk$. The isotropy subgroup $Sp(n)$ must also fix the tangent 3-plane $P_x$, so $P_x$ must be spanned by $xi,xj,xk$. But this argument applies equally well to the 3-plane $Q_x$ tangent to the Hopf fiber through $x$. So $P_x$ coincides with $Q_x$. Thus $\F$ is identical to the Hopf fibration.
    \end{proof}
    \begin{thm}\label{Thm:15Sphere}
      Let $\F$ be a fiberwise homogeneous $C^1$ fibration of the $15$-sphere by 7-spheres. Then $\F$ is the Hopf fibration.
    \end{thm}
    \begin{proof}
      Let $G$ be a subgroup of $SO(16)$ which acts transitively on $\F$. Let $H$ be the normal subgroup of $G$ which takes each fiber of $\F$ to itself. Then $G/H$ acts transitively and effectively on the base of $\F$. By Proposition~\ref{Prop:BaseDiff}, the base of $\F$ is diffeomorphic to $S^8$. By Proposition~\ref{Prop:EffAct}, $G/H$ is isomorphic to $SO(9)$.

      By Lemma~\ref{Lem:Irred}, $G$ acts irreducibly on $\R^{16}$, and by the Key Lemma, $G$ contains a subgroup isomorphic to $Spin(9)$ which acts as the spin representation on $\R^{16}$.

      Let $x\in S^{15}$ be arbitrary, and let $F_x$ be the fiber of $\F$ through $x$. Let $P_x$ be the tangent 7-plane to $F_x$ through $x$. The isotropy subgroup of $Spin(9)$ which fixes $x$ is $Spin(7)$, and it acts on the orthogonal $\R^{15}$ as the sum of the $SO(7)$ action on an $\R^7$ and the spin representation on an $\R^8$, see Ziller \cite{ziller1982homogeneous}. The tangent 7-plane $P_x$ must be fixed by the isotropy action, and hence must lie in the $\R^7$. But this argument applies equally well to the 7-plane $Q_x$ tangent to the Hopf fiber through $x$. So $P_x$ coincides with $Q_x$. Thus $\F$ is identical to the Hopf fibration.
    \end{proof}
    This concludes the proof of the Main Theorem. Those readers who are interested in the proof of the Key Lemma can read Section~\ref{Sec:RepnTheory}. Those who are not may safely ignore it.
    
  \section{Representation Theory}\label{Sec:RepnTheory}
    We have saved the proof of the Key Lemma for this section, so that we can black-box the representation theory of compact Lie groups. The proof is totally standard. Our use for it is unique enough that we are unlikely to find its exact statement in the literature. The main tool we will use is a comprehensive list of low-dimensional irreducible representations of the classical compact Lie groups, due to Andreev-Vinberg-\'Elashvili \cite{andreev1967orbits}. But we will also need to deal with the minor irritation that their list is of complex representations, and for our purposes we need to know about real irreducible representations. We will also need to say a little about irreducible representations of product groups.
    
    We make use of the fact that complex irreducible representations of a compact Lie group are of ``real type'' or ``complex type'' or ``quaternionic type.'' We leave as a black box the precise meanings of these terms. See Br\"ocker and tom Dieck \cite{brocker1985representations} for more details.
    
    Let $G$ be a compact Lie group, and let $\rho:G\to GL(\R^n)$ be an irreducible representation. Let $\rho_{\C}:G\to GL(\C^n)$ be $\rho$ followed by the natural inclusion $GL(\R^n)\hookrightarrow GL(\C^n)$. Conversely, let $\pi:G\to GL(\C^n)$ be a complex irreducible representation, and let $\pi_{\R}:G\to GL(\R^{2n})$ be the result of forgetting the complex structure on $\pi$.
    \begin{prop}[Theorem 6.3 in \cite{brocker1985representations}]\label{Prop:RealIrrep}
      Let $\rho:G\to GL(n,\R)$ be irreducible. One of two possibilities holds:
      \begin{enumerate}
        \item $\rho_{\C}$ is a complex irreducible representation of $G$ of dimension $n$, of real type.
        \item $\rho$ is equal to $\pi_{\R}$ for some complex irreducible representation $\pi$ of $G$ of dimension $n/2$, of complex or quaternionic type.
      \end{enumerate}
    \end{prop}
    \begin{prop}[Lemma 4 and Table 1 in \cite{andreev1967orbits}, Theorems IX.10.5-7 in \cite{simon1996representations}]\label{Prop:IrrepList}
      Let $G$ be a compact simple Lie group. What follows is a complete list of the irreducible representations of $G$ in $GL(\C^d)$ with dimension $d$ less than $\dim G$, and their type (i.e.~real or complex or quaternionic).
      \begin{itemize}
        \item $G=SU(n+1), \dim G=n^2+2n$. Let $n \geq 2$. One irreducible representation each of dimension $n+1$ (the defining representation), $(n+1)(n+2)/2$, $n(n+1)/2$. They're all of complex type, except that when $n=3$, the last is of real type. When $n=2$, the last one is equivalent to the defining representation. When $n=5,6,7$, there is an additional representation of dimension $(n-1)n(n+1)/6$, which is quaternionic for $n=5$, and of complex type otherwise.
        \item $G=Sp(n), \dim G=n(2n+1)$. One irreducible representation of dimension $2n$ of quaternionic type, and one of dimension $2n^2-n-1$ of real type. When $n=3$, there's an additional one of dimension 14 of quaternionic type.
        \item $G=Spin(9), \dim G=36$. One irreducible representation of dimension 9, one of dimension 16, both of real type.
      \end{itemize}
    \end{prop}
    
    \begin{prop}[Theorem 3.9 in Sepanski \cite{sepanski2007compact}]\label{Prop:TensorProduct}
      Let $G$ and $H$ be compact Lie groups. A representation of $G\times H$ in $GL(n,\C)$ is irreducible if and only if it is the tensor product of an irreducible representation of $G$ with one of $H$.
    \end{prop}
    \begin{prop}[follows from Theorem 5.22 in Sepanski \cite{sepanski2007compact}]\label{Prop:ProdGroup}
      Let $H$ be a normal (possibly disconnected) subgroup of a compact connected Lie group $G$, and suppose $K:=G/H$ is simple. Then there exist finite-sheeted covering groups $G'$, $H'$ of $G$, $H_0$ (the identity component of $H$), such that $G'\cong H'\times \tilde{K}$, where $\tilde{K}$ is the universal covering group of $K$.
    \end{prop}
    Let $V$ be a complex irreducible representation of a compact Lie group $G$. Pick an arbitrary basis for $V$, so that we identify $V$ with $\C^n$, and identify $G$ with a subgroup of $GL(n,\C)$. The character of $V$ is a function $\mychi_V:G\to \C$, defined by
    \[ \mychi_V(g) = \Tr(g). \]
    The trace $\Tr(g)$ is independent of our choice of identification of $V$ with $\C^n$, so $\mychi_V$ is well-defined.
    \begin{prop}[Proposition 6.8 in Br\"ocker and tom Dieck \cite{brocker1985representations}]\label{Prop:RCQformula}
      Let $V$ be a complex irreducible representation of a compact Lie group $G$ with character \mbox{$\mychi_V:G\to \C$}. Then
        \[ \int_G \mychi_V(g^2) dg = \left\{
          \begin{array}{rcl}
            1 & \Leftrightarrow & \mbox{$V$ is of real type} \\
            0 & \Leftrightarrow & \mbox{$V$ is of complex type} \\
            -1 & \Leftrightarrow & \mbox{$V$ is of quaternionic type} \\
          \end{array}
        \right.
        \]
    \end{prop}
    \begin{cor}\label{Cor:TensorType}
      Let $V, W$ be complex irreducible representations of $G, H$ respectively. Let $V \otimes W$ be the tensor product of $V$ and $W$, an irreducible representation of $V\times W$ (see Proposition~\ref{Prop:TensorProduct}). Then
      \begin{itemize}
        \item $V \otimes W$ is of real type if $V$ and $W$ are both of real type or both of quaternionic type.
        \item $V \otimes W$ is of complex type if at least one of $V$ and $W$ are of complex type.
        \item $V \otimes W$ is of quaternionic type if one of $V$ and $W$ is of real type and the other is of quaternionic type.
      \end{itemize}
    \end{cor}
    \begin{proof}
      Observe that $\mychi_{V\otimes W}(g,h) = \mychi_V(g)\mychi_W(h)$. Then
      \begin{align*}
        \int_{G\times H} \mychi_{V\otimes W}(g^2,h^2)dgdh &= \int_{G\times H} \mychi_V(g^2)\mychi_W(h^2) dgdh \\
          &= \left( \int_G \mychi_V(g^2) dg \right) \left( \int_H \mychi_W(h^2) dh \right).
      \end{align*}
      The result follows.
    \end{proof}
    \keylemma
    \begin{proof}
      In each of the three cases, the strategy is the same. We want to show that $G$ is a simply connected Lie group acting in the standard way. We use Proposition~\ref{Prop:ProdGroup} to find a covering group $G'$ of $G$ which we can write as a product. We pull back the action of $G$ to an action of $G'$, and use Proposition~\ref{Prop:TensorProduct} to decompose this action as a tensor product of irreducible representations. Then we use the list of low-dimensional representations in Proposition~\ref{Prop:IrrepList}, together with Corollary~\ref{Cor:TensorType} to show that the action of $G'$ is standard, and in particular is nontrivial on its center, so that $G=G'$.
      \begin{enumerate}
        \item There exist covering groups $G'$, $H'$ of $G$, $H_0$ for which $G'=H' \times SU(n+1)$. The action of $G$ pulls back to an irreducible action of $G'$. By Proposition~\ref{Prop:RealIrrep}, there is either a complex irreducible representation of dimension $2n+2$ and of real type which restricts to the real action of $G'$, or there's a complex irreducible representation of dimension $n+1$ of complex or quaternionic type which equals the real action of $G'$ after forgetting the complex structure.
        
        Either way, the complex action of $G'$ is the tensor product of irreducible representations of $H'$ and $SU(n+1)$, by Proposition~\ref{Prop:TensorProduct}, and the latter is nontrivial (or else $G/H$ could not act transitively on the base of $\F$). Suppose first that $n \geq 2$; we'll return to $n=1$ momentarily. When $n \geq 2$, we have $2n+2 < n^2 + n$, and so we may find every complex irreducible representation of $SU(n+1)$ of dimension at most $2n+2$ on the list found in Proposition~\ref{Prop:IrrepList}, case 1.
        
        The only irreducible representations of $SU(n+1)$ not of complex type are one of real type in dimension 6 when $n=3$, and one of quaternionic type of dimension 20 when $n=5$. These dimensions do not divide either $n+1$ or $2n+2$ in either case. By Corollary~\ref{Cor:TensorType}, the action of $G'$ must be the result of forgetting the complex structure on the $(n+1)$-dimensional tensor product of a complex irreducible representation of $H'$ with a complex irreducible representation of $SU(n+1)$ of complex type. It follows that we must be looking at the tensor product of a \mbox{1-dimensional} representation of $H$ with the defining representation of $SU(n+1)$, because every other representation of $SU(n+1)$ has dimension larger than $n+1$.
        
        If $n=1$, then the action of $G'$ on $\R^4$ is either the restriction of an irreducible complex action on $\C^4$ of real type or it's an irreducible action on $\C^2$, forgetting the complex structure, and of complex or quaternionic type. The group $SU(2)$ has precisely one complex irreducible representation of dimension 2 (the defining representation), and one of dimension 4, and they're both of quaternionic type. But if we take the tensor product of the 4-dimensional representation of $SU(2)$ with a 1-dimensional representation of $H'$, then the result cannot be of real type, applying Corollary~\ref{Cor:TensorType} and observing that 1-dimensional representations of compact Lie groups are never of quaternionic type. (For semi-simple compact Lie groups, their only 1-dimensional representation is the trivial one, and thus is of real type, and for tori, the 1-dimensional representations are always of complex type.)
        
        Thus the action of $G'$ always has a subgroup the defining representation of $SU(n+1)$. This action is nontrivial on its center, so when the action of $SU(n+1)$ descends to the original action of $G$, we keep the full $SU(n+1)$ group. Thus $G$ contains $SU(n+1)$ as a subgroup acting in the standard way.
        
        \item There exist covering groups $G'$, $H'$ of $G$, $H_0$ for which $G'=H'\times Sp(n+1)$. The action of $G$ pulls back to an action of $G'$. By Proposition~\ref{Prop:RealIrrep}, there's either a complex irreducible representation of $G'$ of dimension $4n+4$ and of real type, restricting to the real action of $G'$, or a complex irreducible representation of $G'$ of dimension $2n+2$ and of complex or quaternionic type which equals the real action of $G'$.
        
        An irreducible representation of $G'$ is the tensor product of irreducible representations of $H'$ and $Sp(n+1)$, by Propositon~\ref{Prop:TensorProduct}. The irreducible representation of $Sp(n+1)$ must be nontrivial, or else $G/H$ could not act transitively on the base of $\F$. It must also have dimension at most $4n+4$, which is less than $2n^2+5n+3$ for $n\geq 1$. Thus the irreducible representation of $Sp(n+1)$ appears in the list in Proposition~\ref{Prop:IrrepList}. The dimension of the representation which is not the defining one is $2(n+1)^2 - (n+1)-1$, or $2n^2+3n$. This quantity is greater than $4n+4$ for $n \geq 2$, and equals $5$ for $n=1$, so it certainly does not divide either $2n+2$ or $4n+4$. So it cannot appear in the tensor product representation of $G'$. Similarly, the other representation of $Sp(3)$ of dimension $14$ cannot appear in the tensor product representation of $G'$, because $14$ does not divide $6$ or $12$.
        
        Thus the irreducible representation of $G'$ of dimension $2n+2$ must contain the defining representation of $Sp(n+1)$ as a subgroup. This representation is nontrivial on its center, so it projects down to the defining representation of $Sp(n+1)$ as a subgroup of $G$.
        
        \item There exist covering groups $G'$, $H'$ of $G$, $H_0$ for which $G'=H'\times Spin(9)$. The action of $G$ pulls back to an action of $G'$. By Proposition~\ref{Prop:RealIrrep}, there's either a complex irreducible representation of $G'$ of dimension 16 and of real type, restricting to the real action of $G'$, or there's a complex irreducible representation of $G'$ of dimension 8 which equals the real action of $G'$. The action of $G'$ is the tensor product of irreducible representations of $H'$ and of $Spin(9)$, by Proposition~\ref{Prop:TensorProduct}, and the representation of $Spin(9)$ is nontrivial or else $G/H$ could not act transitively on the base of $\F$.
        
        There are only two nontrivial complex irreducible representations of $Spin(9)$ of dimension less than 36; the vector representation of dimension 9 and the spin representation of dimension 16. Thus the action of $G'$ is the tensor product of a 1-dimensional representation of $H'$ with the 16-dimensional spin representation of $Spin(9)$. This action is nontrivial on the center of $Spin(9)$, and so $G$ contains a copy of $Spin(9)$ as well, acting as the spin representation on $\R^{16}$.
      \end{enumerate}
      This concludes the proof of the Key Lemma.
    \end{proof}
    
  \section{Locally fiberwise homogeneous fibrations in the 3-sphere}\label{Sec:Local}
    \subsection{Background}
      A stronger version of the Main Theorem would be the following local version: if we fiber a connected open subset of a round sphere by smooth subspheres, so that it's {\em locally fiberwise homogeneous} | so that any two fibers have open sets around them with an isometry taking one open set to the other, preserving fibers and taking the first given fiber to the second | then that fibration is a portion of a Hopf fibration.
    
      The proof that a global fiberwise homogeneous fibration of a round sphere by subspheres is a Hopf fibration relies heavily on the global structure of the fibration. To even get off the ground in proving the statement, we need to know exactly which manifold the base space of the fibration is diffeomorphic to. So it does not appear that we can prove our local theorem with the same methods we used to prove the global theorem.
      
      In this chapter we prove the following theorem.
      \begin{thm}\label{Thm:LocallyFWH}
        Let $\F$ be a locally fiberwise homogeneous fibration of a connected open set in the 3-sphere by great circles. Then $\F$ is a portion of a Hopf fibration.
      \end{thm}
      We prove our desired local theorem only for the lowest dimension, and only under the restriction that our fibers are great circles. The reason we can prove our theorem under these circumstances is that we can make use of a moduli space for the space of great-circle fibrations of the 3-sphere.
      
    \subsection{Great circle fibrations of the 3-sphere}
      The following description of great circle fibrations of the 3-sphere is due to Herman Gluck and Frank Warner \cite{gluck1983great}.
    
      An oriented great circle in the 3-sphere can be identified with an oriented 2-plane through the origin in $\R^4$. Thus a great circle in the 3-sphere is identified with a point in $\Gr$, the Grassmann manifold of oriented 2-planes in $\R^4$. A fibration of the 3-sphere by great circles is identified in this way with a 2-dimensional submanifold of $\Gr$. We also may identify $\Gr$ with the manifold $S^2\times S^2$.
    
      \begin{thm}[From \cite{gluck1983great}]\label{Thm:GW}
        There is a one-to-one correspondence between great circle fibrations of the 3-sphere and submanifolds of $\Gr = S^2 \times S^2$ which are graphs of distance decreasing functions from one $S^2$ factor to the other.
      
        Similarly, there is a one-to-one correspondence between great circle fibrations of open sets in the 3-sphere, and submanifolds of $\Gr = S^2\times S^2$ which are graphs of distance decreasing functions from an open set in one $S^2$ factor to the other factor.
      
        The one-to-one correspondence is simply to identify a great circle with a point in $\Gr$.
      \end{thm}
      
      This moduli space for great circle fibrations of the 3-sphere allows us to more easily answer questions we have about these fibrations, by translating them into questions about distance-decreasing functions on the 2-sphere. In \cite{gluck1983great}, Gluck and Warner use this method to show (for example) that any great circle fibration of the 3-sphere contains two orthogonal fibers, and that the space of such fibrations deformation retracts to the subspace of Hopf fibrations.
      
      We need to translate the ideas of ``fiberwise homogeneous'' and ``locally fiberwise homogeneous'' from the setting of fibrations to the setting of distance-decreasing maps on the 2-sphere. In the former setting, we have a subgroup of $SO(4)$ acting on a fibration in the 3-sphere; therefore, we need to know how the $SO(4)$ action on the 3-sphere translates to the setting of $S^2\times S^2$, where the distance-decreasing functions live.
      
      The group $SO(4)$ is double covered by $SU(2)\times SU(2)$, and in turn double covers $SO(3)\times SO(3)$. If we follow the identification of $\Gr$ with $S^2 \times S^2$ (see \cite{gluck1983great}) for details), we find that the action of $SO(4)$ on $S^3$ induces an action of $SO(3)\times SO(3)$ on $S^2\times S^2$, where the first (respectively second) $SO(3)$ factor acts by isometries on the first (resp.~second) $S^2$ factor.

      We call a subset $S$ of $S^2\times S^2$ homogeneous if some subgroup of the isometry group of $S^2\times S^2$ preserves $S$ and acts transitively on it. We say that $S$ is locally homogeneous if for any $s_1, s_2\in S$ there is an isometry of $S^2\times S^2$ taking a neighborhood of $s_1$ isometrically to a neighborhood of $s_2$.
      \begin{prop}
        Let $\F$ be a fibration of (an open set in) $S^3$ by great circles. Let $S$ be the corresponding graph in $S^2\times S^2$ given by Theorem~\ref{Thm:GW}; i.e. identify each great circle with a point in $\Gr$. Then $\F$ is (locally) fiberwise homogeneous if and only if $S$ is (locally) homogeneous in $S^2\times S^2$.
      \end{prop}
      \begin{proof}
        The proof is immediate from the definitions of fiberwise homogeneous, locally fiberwise homogeneous, homogeneous, locally homogeneous, and the identification of great circles with points in \mbox{$\Gr=S^2\times S^2$}.
      \end{proof}
      
    \subsection{Locally fiberwise homogeneous fibrations are subsets of Hopf fibrations}
      We now prove the main theorem of this section.
      \begin{thm:local}
       Let $\F$ be a locally fiberwise homogeneous fibration of a connected open set in the 3-sphere by great circles. Then $\F$ is a portion of a Hopf fibration.
      \end{thm:local}
      
      The distance-decreasing functions associated to the Hopf fibrations are the constant functions. We will show that that the distance-decreasing map $f$ of $S^2$ associated to the fibration is the constant map. We will assume that this map $f$ is at least $C^2$.
      
      \begin{proof}
        Let $\F$ be a locally fiberwise homogeneous fibration of an open set $W\subseteq S^3$ by great circles. Let $f:V\to S^2$ be the corresponding distance-decreasing function on an open set $V\subseteq S^2$ whose graph in $S^2\times S^2$ consists of the great circles making up $\F$. We assume that $f$ is $C^2$. As a consequence, the differential $df_x$ of $f$ is defined at every point $x\in V$. The homogeneity of the graph of $f$ is equivalent to the following: for every $x$ and $x'$ in $V$, there are isometries $g_1$ and $g_2$ of $S^2$ such that $g_1$ takes an open neighborhood $U$ of $x$ to an open neighborhood $U'$ of $x'$ (and takes $x$ to $x'$), and such that
        \[ f \circ g_1 = g_2 \circ f \]
        holds on $U$.
    
        The homogeneity of $f$ implies that the differential $df$ is ``similar'' independent of $x\in V$. We would like to say something along the lines of: the homogeneity of the graph of $f$ implies that the eigenvalues of $df_x$ are independent of $x$. After all, $df_x$ is a linear map between 2-dimensional vector spaces. But $df_x$ is not a map from a vector space to itself. So instead, we consider the image under $df_x$ of the unit circle in the tangent space to $x$. The image of a circle centered at the origin under a general linear map will be some ellipse, possibly degenerate (possibly a circle or line segment or point), and when the ellipse has distinct axes, the preimages of the axes will be orthogonal.
    
        The homogeneity of $f$ implies that the image under $df_x$ of the unit circle in the tangent space to $x$ will be independent of $x$. That is, the magnitude of the major and minor axes of the ellipse will be constant (and possibly identical and/or zero) for all $x\in V$.
      
        The image of a unit circle under $df$ is either a circle (possibly with radius 0) or a proper ellipse with distinct major and minor axes. We show that the second possibility cannot happen; the axes must be identical.
      
        Suppose that the ellipse has two different axes. We will derive a contradiction. In that case, the local homogeneity of $f$ implies not only that the magnitudes of the axes of the ellipses are constant, but also that the local isometries preserve the preimages of the major and minor axes. In other words, let $X,Y$ be unit vector fields along $V$ which map via $df$ to the major and minor axes of the ellipses in the tangent spaces of $f(V)$. Then the local isometries which commute with $f$ also preserve $X$ and $Y$. Now the following lemma applies to show that $V$ must have nonpositive curvature.
      
        \begin{lem}\label{Lem:NonposCurv}
          Let $F$ be a locally homogeneous surface, and suppose $(X,Y)$ is an orthonormal frame along $F$ which is preserved by the locally homogeneous structure (i.e.~the isometries which take any point of $F$ to any other also preserve $X$ and $Y$). Then the sectional curvature of $F$ is nonpositive.
        \end{lem}
      
        We save the proof for the end of this section. We have a contradiction, because Lemma~\ref{Lem:NonposCurv} tells us $V$ has nonpositive curvature, yet $V$ is an open subset of the round sphere, which has positive curvature. Thus the image under $df$ of the unit circle in a tangent space to $V$ must be a circle of radius $r \geq 0$, with $r$ independent of the point in $V$. Note first that $r \leq 1$, because $f$ is distance-decreasing. In fact we must have that $r$ is strictly less than 1. Even though distance-decreasing functions might preserve distance infinitesimally, if $r=1$ then $f$ preserves all geodesic distances and hence is not distance-decreasing. Now we show that we must have $r=0$. Suppose $r>0$. Then $f$ multiplies all geodesic distances by $r$, and hence multiplies curvature by $1/r^2$, which is greater than 1. But the image of $V$ lies in a 2-sphere of the same radius as the domain, so this is impossible. Therefore we must have $r=0$, and hence $df\equiv 0$. It follows that $f$ is a constant map, because $V$ is connected, and therefore our locally fiberwise homogeneous fibration is a portion of a Hopf fibration.
      \end{proof}
      
      \begin{proof}[Proof of Lemma~\ref{Lem:NonposCurv}]
        We denote sectional curvature of a plane spanned by an orthonormal basis by $K(X,Y) = \langle R(X,Y)X, Y \rangle$, where $R$ is the Riemannian curvature tensor, and we denote by $\n$ the Riemannian connection on $F$. Note that any real-valued function on $F$ depending only on $X$ and $Y$ is constant, because of the local homogeneous structure on $F$ preserving $X$ and $Y$. We have
        \begin{align*}
        K(X,Y) &= \langle R(X,Y)X, Y \rangle \\
            &= \langle \n_Y \n_X X - \n_X \n_Y X + \n_{[X,Y]} X, Y \rangle \\
            &= \langle \n_Y \n_X X, Y \rangle
               + \langle \n_X \n_Y X, Y \rangle
               + \langle \n_{[X,Y]} X, Y \rangle \\
            &= Y \langle \n_X X, Y \rangle - \langle \n_X X, \n_Y Y \rangle \\
              &\quad + X \langle \n_Y X, Y \rangle - \langle \n_Y X, \n_X Y \rangle \\
              &\quad + \langle \n_{[X,Y]} X, Y \rangle \\
            &= 0 + 0 + \langle \n_{[X,Y]} X, Y \rangle .
        \end{align*}
        In the fourth equality, we use the compatibility of the Riemannian connection with the metric. In the fifth equality, we use the knowledge that functions of $X$ and $Y$ are constant to show that the first and third terms are equal to $0$. We find that the second and fourth terms are equal to $0$ by computing that $\n_X X$ points along $Y$ and $\n_Y Y$ points along $X$; similarly $\n_Y X$ points along $Y$ and $\n_X Y$ points along $X$. A sample computation along these lines is:
          \[ \langle \n_X X, X \rangle = \frac{1}{2}X\langle X, X \rangle = 0, \]
        and hence $\n_X X$, being orthogonal to $X$, points along $Y$. The other computations are similar.
          
        Now that we have shown $K(X,Y)= \langle \n_{[X,Y]} X, Y \rangle$, we will show that the latter is nonpositive. The bracket $[X,Y]$ is preserved by the local isometries of $F$, so we can write $[X,Y]=aX+bY$ for constant $a,b$. Then,
        \begin{align*}
          K(X,Y) &= \langle \n_{[X,Y]} X, Y \rangle \\
            &= \langle \n_{aX + bY} X, Y \rangle \\
            &= a\langle \n_X X, Y \rangle + b\langle \n_Y X, Y \rangle.
        \end{align*}
        Using compatibility of the connection with the metric again, together with symmetry of the connection and the fact that functions of $X$ and $Y$ are constant, we quickly arrive at:
        \begin{align*}
          \langle \n_X X, Y \rangle &= -\langle X, [X,Y] \rangle = -\langle X, aX+bY \rangle = -a, \\
          \langle \n_Y X, Y \rangle &= -\langle Y, [X,Y] \rangle = -\langle Y, aX+bY \rangle = -b,
        \end{align*}
        from which it follows that
          \[ K(X,Y) = a\langle \n_X X, Y \rangle + b\langle \n_Y X, Y \rangle = -a^2-b^2 \leq 0.\]
        So the curvature of $F$ is nonpositive.
      \end{proof}

\nocite{wong1961isoclinic, wolf1963elliptic, wolf1963geodesic, escobales1975riemannian, gromoll1985one, ranjan1985riemannian, gluck1986geometry, gluck1987fibrations}

\bibliography{thesis}{}
\bibliographystyle{amsplain}
\end{document}